\documentclass[12pt]{amsart}
\usepackage{latexsym,amssymb,amsmath,tikz, xypic}
\usepackage{comment}

\numberwithin{equation}{section}
\newtheorem{theorem}{Theorem}[section]
\newtheorem{lemma}[theorem]{Lemma}
\newtheorem{proposition}[theorem]{Proposition}
\newtheorem{corollary}[theorem]{Corollary}

\theoremstyle{definition}
\newtheorem{definition}[theorem]{Definition}

\def\NZQ{\mathbb}               
\def\NN{{\NZQ N}}

\def\frk{\mathfrak}               

\def\Phi{{\frk N}}
\def\G{\Gamma}

\def\opn#1#2{\def#1{\operatorname{#2}}} 
\opn\chara{char} 
\opn\length{\ell} 
\opn\pd{pd} 
\opn\rk{rk}
\opn\projdim{proj\,dim} 
\opn\injdim{inj\,dim} 
\opn\rank{rank}
\opn\depth{depth} 
\opn\grade{grade} 
\opn\height{height}
\opn\embdim{emb\,dim} 
\opn\codim{codim}

\opn\Tr{Tr} 
\opn\bigrank{big\,rank}
\opn\superheight{superheight}
\opn\lcm{lcm}
\opn\trdeg{tr\,deg}
\opn\reg{reg} 
\opn\lreg{lreg} 
\opn\ini{in} 
\opn\lpd{lpd}
\opn\size{size}
\opn\mult{mult}
\opn\dist{dist}
\opn\cone{cone}
\opn\lex{lex}
\opn\rev{rev}
\opn\div{div} \opn\Div{Div} \opn\cl{cl} \opn\Cl{Cl}
\opn\Spec{Spec} \opn\Supp{Supp} \opn\supp{supp} \opn\Sing{Sing}
\opn\Ass{Ass} \opn\Min{Min}
\opn\Ann{Ann} \opn\Rad{Rad} \opn\Soc{Soc}
\opn\Syz{Syz} \opn\Im{Im} \opn\Ker{Ker} \opn\Coker{Coker}
\opn\Am{Am} \opn\Hom{Hom} \opn\Tor{Tor} \opn\Ext{Ext}
\opn\End{End} \opn\Aut{Aut} \opn\id{id} \opn\ini{in}

\opn\nat{nat}
\opn\pff{pf}
\opn\Pf{Pf} \opn\GL{GL} \opn\SL{SL} \opn\mod{mod} \opn\ord{ord}
\opn\Gin{Gin}
\opn\Hilb{Hilb}\opn\adeg{adeg}\opn\std{std}\opn\ip{infpt}
\opn\Pol{Pol}
\opn\sat{sat}
\opn\Var{Var}
\opn\Gen{Gen}
\opn\Graph{Graph}
\opn\p{p}
\opn\d{d}
\opn\r{r}
\opn\m{m}
\opn\im{im}
\opn\SimpComp{SimpComp}

\opn\aff{aff} \opn\con{conv} \opn\relint{relint} \opn\st{st}
\opn\lk{lk} \opn\cn{cn} \opn\core{core} \opn\vol{vol}
\opn\link{link} \opn\star{star}
\opn\gr{gr}

\def\pot#1#2{#1[\kern-0.28ex[#2]\kern-0.28ex]}

\opn\dirlim{\underrightarrow{\lim}}
\opn\inivlim{\underleftarrow{\lim}}
\begin{document}
	

\title{Tuples of homological invariants of edge ideals}
\thanks{\today}

\author[A. Kanno]{Akane Kanno}
\address{Department of Pure and Applied Mathematics, Graduate School
	of Information Science and Technology, Osaka University, Suita, Osaka
	565-0871, Japan}
\email{u825139b@ecs.osaka-u.ac.jp}


\begin{abstract}
	
	Let $G$ be a graph and $I(G)$ its edge ideal. In this paper, we completely determine the tuples $(\dim R/I(G), \depth (R/I(G)), \reg (R/I(G)))$ when the number of vertices is fixed for any graphs $G$. 
\end{abstract}

\maketitle

\section{introduction}

In this paper, graphs are always assumed to be finite, simple, undirected and connected unless otherwise noted. Let $R = K[x_1, \dots, x_n]$ denote a standard graded polynomial ring over a field $K$ and $G$ a graph on vertex set $[n] = \{1, \dots, n\}$ with edge set $E(G)$. The edge ideal of $G$ is the ideal of $R$ generated by monomials $x_{i}x_{j}$ where $\{i, j\} \in E(G)$. Denote the edge ideal as $I(G)$, and denote 
dimension, depth, regularity and $h$-polynomial of $R/I(G)$ as $\dim G, \depth G, \reg G$ and $h_G$, respectively. 

In recent years, a major trend in edge ideals has been the investigation of not only the relationship between invariants of graphs (e.g., matching numbers, induced matching numbers) and invariants of edge ideals but also the relationship among invariants of edge ideals.
The following inequalities are well known as one of the most basic relationship between themselves (\cite[Corollary B.4.1]{V2}).

\begin{displaymath}
	\deg h_G - \reg G
	\leq \dim G - \depth G
\end{displaymath}

In the above inequality, the equality holds if and only if the last Betti number $\beta_{p, p+r}(R/I(G))$ is nonvanishing where $p = \projdim(R/I(G))$, $r = \reg G$. In particular, if $R/I(G)$ is Cohen-Macaulay, then both sides are zeros and the equality holds. 
Based on this formula, the study has been done to determine the range that the above invariants can take when the number of vertices of the graph is fixed, i.e., to describe the sets defined as follows.

\begin{definition}
	
Given a positive integer $n$, we define the following three sets: 
	\begin{displaymath}
		\begin{aligned}
			\Graph_{\dim, \depth}(n) &= \left\{(d, p) \in \NN^2 \ \middle| \ \text{there is a graph }
			G\text{ such that }\dim G = d, \depth G = p \right\}, \\
			\Graph_{\dim, \depth, \reg}(n) &= \left\{(d, p, r) \in \NN^3 \ \middle| \ 
			\begin{array}{c}
				\mbox{\text{there is a graph }$G$\text{ such that }} \\
				\mbox{$\dim G = d, \depth G = p, \reg G = r$} \\ 
			\end{array}
			\right\}, \\ 
			\Graph_{\dim, \depth, \reg, \deg}(n) &= \left\{(d, p, r, g) \in \NN^4 \ \middle| \ \begin{array}{c}
				\mbox{\text{there is a graph }$G$\text{ such that }} \\
				\mbox{$\dim G = d, \depth G = p, \reg G = r, \deg h_G = g$} \\ 
			\end{array}\right\}. 
		\end{aligned}
	\end{displaymath}
\end{definition}

In \cite{HKMVT}, the following two partial results on this issue are obtained. 

\begin{theorem}\cite[Theorem $2.8$]{HKMVT}
	
	Define $C^{*}(n)$ as follows: 
	\[
	C^{*}(n) = \{(d, p) \in \NN^2 \ | \ 1 \le d \le p \le n - 1, d \le (n - d)(d - p + 1)\}. 
	\]
	Then $C^{*}(n) \subset \Graph_{\dim, \depth}(n)$ for any $n \ge 2$. 
\end{theorem}

\begin{theorem}\label{cwdrdd}\cite[Theorem $4.4$]{HKMVT}
	
	Let $n \geq 5$ be an integer. Then 
	\begin{displaymath}
		\begin{aligned}
			{\rm Graph}^{CW}_{{\rm depth}, {\rm reg}, \dim, \deg} (n) 
			&= {\rm CW}_{2, {\rm reg}, \dim, \deg} (n) \\
			&\cup \left\{ (a, d, d, d) \in \mathbb{N}^{4} \  \middle| \
			3 \leq a \leq d \leq \left\lfloor\frac{n - 1}{2}\right\rfloor, \  
			n < a + 2d \right\} \\ 
			&\cup \left\{ (a, a, d, d) \in \mathbb{N}^{4} \  \middle| \
			3 \leq a < d \leq n - a,\   n \leq 2a + d - 1 \right\} \\
			&\cup \left\{(a,r,d,d) \in \mathbb{N}^{4} ~\left|~
			\begin{array}{c}
				\mbox{$3 \leq a < r < d < n - r$, } \\
				\mbox{$n + 2 \leq a + r + d$} \\ 
			\end{array}
			\right \}\right., 
		\end{aligned}
	\end{displaymath}
	where
	\begin{displaymath}
		\begin{aligned}
			&{\rm CW}_{2, {\rm reg}, \dim, \deg} (n) \\
			&= \left\{
			\begin{alignedat}{3}
				&\{(2, 2, n - 2, n - 2), (2, 2, n - 3, n - 3)\}, 
				&\quad &\text{if $n$ is even}, \\
				& \left\{(2, 2, n - 2, n - 2), (2, 2, n - 3, n - 3),
				\left(2, \frac{n-1}{2}, \frac{n-1}{2}, \frac{n-1}{2}\right) \right\}, 
				&\quad &\text{if $n$ is odd}, 
			\end{alignedat}
			\right. 
		\end{aligned}
	\end{displaymath}
	and 
	
	${\rm Graph}^{CW}_{{\rm depth}, {\rm reg}, \dim, \deg} (n)$ is restriction of $\Graph_{\dim, \depth, \reg, \deg}(n)$ to Cameron-Walker graphs. 
\end{theorem}

Furthermore, the following results were obtained in \cite{HKU}. 

\begin{theorem}\cite[Corollary $1.4$]{HKU}
	
	The equality $C^*(n) = \Graph_{\dim, \depth}(n)$ holds if $n \leq 12$. 
\end{theorem}

\begin{theorem}\cite[Theorem $1.5$]{HKU}
	
	Let $n \geq 2$. Then we have $C^*(n) = \Graph^{chordal}_{\dim, \depth}(n)$
	where $\Graph^{chordal}_{\dim, \depth}(n)$ is restriction of $\Graph_{\dim, \depth}(n)$ to chordal graphs. 
\end{theorem}

Generalizing these results, $\Graph_{\dim, \depth, \reg}(n)$ and $\Graph_{\dim, \depth}(n)$ are completely determined in this paper.

\begin{theorem}\label{main}
	
	For all $n \ge 3$, the following holds: 
	\[
	\Graph_{\dim, \depth, \reg}(n) = C^{**}(n), 
	\]
	where
	\begin{displaymath}
		\begin{aligned}
			C^{**}(n) = &\left\{(n - 1, 1, 1)\right\} \cup\\
			 &\left\{(d, p, r) \in \NN^3 \ \middle| \
			 \begin{array}{c}
			 	\mbox{$1 \le p \le d \le n - 2, 2 \le r + d \le n - 1$, }\\
			 	\mbox{$1 \le r \le d \le (n - d - (r - 1))(d - p + 1) + (r - 1)$}\\
			 \end{array}\right\}
		\end{aligned}
	\end{displaymath}
\end{theorem}

\begin{corollary}\label{maina}
	
	For all $n \ge 3$, the following holds: 
	\[
	\Graph_{\dim, \depth}(n) = C^{*}(n). 
	\]
\end{corollary}

\section{preliminaries}

First we prepare materials from graph theory and discuss their properties. Moreover we introduce a relationship of them and ring-theorical invariants of edge ideals. 

\begin{definition}
	Let $G = (V(G), E(G))$ be a graph and $v$ a vertex. We call $S \subset V(G)$ an \textit{independent set} of $G$ if $\{v,v'\} \not\in E(G)$ for any $v,v' \in S$. Moreover we define $\m(G)$, $\im(G)$, $N_G(v)$, $N_G[v]$, $\d(G)$ and $\p(G)$ as follows: 
	\begin{displaymath}
		\begin{aligned}
			\m(G) &= \max\{\lvert M \rvert \ | \ M \subset E(G), e \cap e' = \emptyset \text{ for any }e, e' \in M\}, \\
			\im(G) &= \max\left\{\lvert M \rvert \ \middle| \ 
			\begin{array}{c}
				\mbox{$M \subset E(G), \text{there is no }e'' \in E(G) \text{ with }$}\\
				\mbox{$e \cap e'' \ne \emptyset, e' \cap e'' \ne \emptyset \text{ for any } e,e' \in M$}\\
			\end{array}\right\}, \\
			N_G[v] &= N_G(v) \cup \{v\}, \\
			\d(G) &= \max\{\lvert S \rvert \ | \  S\text{ is a maximal independent set of }G\}, \\
			\p(G) &= \min\{\lvert S \rvert \ | \ S\text{ is a maximal independent set of }G\}. 
		\end{aligned}
	\end{displaymath}
\end{definition}

\begin{lemma}\label{depmax}
	For any graph $G$ and its vertex $v$, the followings hold. 
	\begin{enumerate}
		\item $\dim G = \d(G)$. 
		\item $\depth G \le \p(G)$. 
		\item $\p(G) \le \p(G - v) + 1$. 
		\item $\p(G) \le \p(G - N[v]) + 1$. 
	\end{enumerate}
\end{lemma}

\begin{proof}
	\begin{enumerate}
		\item[(1), (2)] Those are well known facts. 
		\item[(3)] Suppose $p = \p(G) > \p(G - v)$. Then there is a maximal independent set $S$ of $G$ such that $\lvert S \rvert = p$, $v \in S$. 
		Thus $\p(G - N[v]) = p - 1$ because $S \setminus \{v\}$ is a maximal independent set of $G - N[v]$ with minimum cardinality.. 
		Therefore $p - 1 \ge \p(G - v)$. 
		
		If $p - 1 > \p(G - v)$ then there exists a maximal independent set $S' \subset V(G - v)$ such that 
		$S' \cap N[v] \ne \emptyset$ because $S'$ is not a subset of $V(G - N[v])$. 
		Thus $\p(G - v) = p - 1 = \p(G) - 1$. 
		\item[(4)] Let $S$ be a maximal independent set of $G - N[v]$ with $\lvert S \rvert = \p(G - N[v])$. Then $S \cup \{v\}$ is also a maximal independent set of $G$. 
	\end{enumerate}
\end{proof}

Second we mention some properties of regularity and relation with other invariants. 

\begin{proposition}\label{regbasic}\cite[Corollary $18.6$]{Peeva}
	Suppose that $0 \to U \to U' \to U'' \to 0$ is a short exact sequence of graded finitely generated 
	$K[x_1, \dots, x_n]$-modules with homomorphisms of degree $0$. Then 
	\begin{enumerate}
		\item If $\reg(U') > \reg(U'')$, then $\reg(U) = \reg(U')$. 
		\item If $\reg(U') < \reg(U'')$, then $\reg(U) = \reg(U'') + 1$. 
		\item If $\reg(U') = \reg(U'')$, then $\reg(U) \le \reg(U'') + 1$. 
	\end{enumerate} 
\end{proposition}

\begin{proposition}\label{reg}\cite[Lemma 2.10]{DHS}
	For any graph $G$ and any vertex $x \in V(G)$, the following holds. 
	\[
	\reg G \in \left\{\reg \frac{R}{(I(G), x)}, \reg \frac{R}{(I(G) : x)} + 1\right\}.  
	\]
\end{proposition}

\begin{theorem}\cite[Theorem 6.7]{HVT}\cite[Lemma 2.2]{Katzman}
	For any graph $G$, the following inequality holds. 
	\[
	\im(G) \le \reg(G) \le \m(G). 
	\]
\end{theorem}

\begin{theorem}\cite[Theorem 11]{TNT}\label{matching}
	Let $G$ be a graph. Then, $\reg(I(G)) = \m(G) + 1$ if and only if each connected component of $G$ is either a pentagon or a Cameron-Walker graph.
\end{theorem}

\begin{theorem}\cite[Theorem 4.1]{HMVT}
	Let G be a graph on n vertices. Then $\deg h_{R/I(G)}(t) + \reg(R/I(G)) \le n$, 
	and $\dim(R/I(G)) + \reg(R/I(G)) \le n$ hold.  
\end{theorem}

Next we introduce the operation on graphs called $S$-suspension. This notion is initiated in \cite{HKM}.  

\begin{definition}
	Let $G = (V(G), E(G))$ be a graph and $S \subset V(G)$ an independent set. The \textit{$S$-suspension} of $G$, denoted by $G^{S}$, is defined as follows. 
	\begin{itemize}
		\item $V(G^{S}) = V(G) \cup \{v\}$, 
		\item $E(G^{S}) = E(G) \cup \{\{v, w\} : w \in V(G) \setminus S\}. $
	\end{itemize}
\end{definition}

The following properties of $S$-suspension are known. 

\begin{lemma}\label{susp}\cite[Lemma $1.2$]{HKMVT}
	Let $G$ be a graph on $V(G) = \{x_1,\ldots,x_n\}$ and
	$G^S$ the $S$-suspension of $G$ for some independent set $S$ of $G$.  
	If $I(G) \subseteq R =K[x_1,\ldots,x_n]$ and
	$I(G^s) \subseteq  R' = K[x_1,\ldots,x_n,x_{n+1}]$
	are the respective edge ideals,
	then
	\begin{enumerate}
		\item[$(i)$] $\dim(R'/I(G^{S})) = \dim R/I(G)$ if $|S| \leq \dim R/I(G) - 1$. 
		\item[$(ii)$] ${\rm depth}(R'/I(G^{S})) = {\rm depth}(R/I(G))$ if $|S| = {\rm depth}(R/I(G)) - 1$. 
		\item[$(iii)$] ${\rm depth}(R'/I(G^{S})) = 1$ if $S = \emptyset$. 
	\end{enumerate} 
\end{lemma}

Finally we introduce Betti splitting and prove a lemma of regularities of edge ideals by using. Denote $\Gamma(I)$ as the minimal system of monomial generators of $I$. 

\begin{definition}\cite[Definition 1.1]{CHA}
	Let $I, I',$ and $I''$ be monomial ideals such that $\G(I)$ is the disjoint union of $\G(I')$ and $\G(I'')$.
	Then $I = I'+I''$ is a \textit{Betti splitting} if 
	\[\beta_{i,j}(I) = \beta_{i,j}(I')+\beta_{i,j}(I'')+\beta_{i-1,j}(I' \cap I'') \hspace{.5cm}~~\mbox{for all $i\in \NN$ and degrees $j$, }\]
	where $\beta_{i,j}(I)$ denotes the $\{i, j\}$-th graded Betti number of $I$.
\end{definition}

\begin{proposition} \label{reg-pd}\cite[Corollary 2.2]{CHA}
	Let $I = I'+I''$ be a Betti splitting.  Then
	\begin{enumerate}
		\item[(a)] $\reg(I) = \max\{\reg I',\reg I'',\reg I' \cap I'' -1\}$, and
		\item[(b)] $\operatorname{pd}(I) = \max\{\operatorname{pd}(I'),\operatorname{pd}(I''),
		\operatorname{pd}(I'\cap I'')+1\}$,
	\end{enumerate}
\end{proposition}

\begin{definition} \label{d.onevar}\cite[Definition 2.6]{CHA}
	Let $I$ be a monomial ideal in $R=k[x_1,\dots,x_n]$. Let $I'$ be the ideal generated by all 
	elements of $\G(I)$ divisible by $x_i$, and let $I''$ be the ideal generated by all other 
	elements of $\G(I)$. We call $I=I'+I''$ an \textit{$x_i$-partition} of $I$. If $I=I'+I''$ 
	is also a Betti splitting, we call $I=I'+I''$ an \textit{$x_i$-splitting}.
\end{definition}

\begin{proposition} \label{bettispli}\cite[Corollary 2.7]{CHA}
	Let $I=I'+I''$ be an $x_i$-partition of $I$ in which all elements of $I'$ are divisible by $x_i$. 
	If $\beta_{i,j}(I' \cap I'') > 0$ implies that $\beta_{i,j}(I')=0$ for all $i$ and multidegrees $j$, 
	then $I=I'+I''$ is a Betti splitting. In particular, if the minimal graded free resolution of $I'$ is 
	linear, then $I=I'+I''$ is a Betti splitting. 
\end{proposition}

The following is a key for the proof of Theorem \ref{main}. 

\begin{lemma}\label{twin}
	Let $G$ be a graph and $v$ a vertex. Define $G'$ as follows. 
	\begin{itemize}
		\item $V(G') = V(G) \cup \{v'\}. $
		\item $E(G') = E(G) \cup \{\{v',w\} \ | \ w \in N(v)\}. $
	\end{itemize}
	Then $\reg(G) = \reg(G'). $
\end{lemma}

\begin{proof}
	There exist the following exact sequences:  
	
	\[
	0 \to I(G) \cap (v'w \ | \ w \in N(v')) \to I(G) \oplus (v'w \ | \ w \in N(v')) \to I(G') \to 0, 
	\]
	
	\[
	0 \to I(G - v) \cap (vw \ | \ w \in N(v)) \to I(G - v) \oplus (vw \ | \ w \in N(v)) \to I(G) \to 0. 
	\]
	
	Denote $T = I(G - v) \cap (vw \in N(v'))$, $T' = I(G) \cap (v'w \in N(v'))$. 
	By virtue of Proposition \ref{bettispli}, the above $I(G) = I(G - v) + (vw \ | \ w \in N(v))$ and $I(G') = I(G) + (v'w \ | \ w \in N(v'))$ are Betti splitting since 
	$(vw \ | \ w \in N(v)) \cong (v'w \ | \ w \in N(v'))$ has a linear resolution. 
	Therefore $\reg(I(G')) = \max\{\reg(T') - 1, \reg(I(G))\}$, 
	$\reg(I(G)) = \max\{\reg(T) - 1, \reg(I(G - v))\}$ by Proposition \ref{reg-pd}. 
	
	Moreover the following sequence is exact. 
	 
	\[
	0 \to (T' : v)(-1) \xrightarrow{\cdot v} T' \to (T', v) \to 0. 
	\]
	
	Here $(T' : v) = (v'w \ | \ w \in N(v'))$ and $(T', v) = (I(G), v) \cap (v, v'w \ | \ w \in N(v')) \cong (T, v')$. 
	Thus $\reg(T') = \reg(T)$ by Proposition \ref{regbasic}. 
	
	Hence $\reg(I(G')) = \max\{\reg(I(G)), \reg(T) - 1\} = \reg(I(G))$. 
	
\end{proof}

\section{proof of main results}

In this section, we prove Theorem \ref{main} and Corollary \ref{maina}. Prior to the proof, one lemma is introduced.

\begin{lemma}\label{regz}\cite{HKMVT}\cite{HKU}
	For any $3 \le n$ and $(p, d) \in C^{*}(n)$ there is a graph with $n$ vertices such that $G$ satisfying $\dim  G = d, \depth G = p, \reg G  = 1$. 
	
\end{lemma}


\begin{proof}[Proof of Theorem \ref{main}]
	
	First, we prove $C^{**}(n) \subset \Graph_{\dim, \depth, \reg}(n)$. 
	
	It is sufficient to construct a graph $G$ with $n$ vertices such that 
	$\dim G = d$, $\depth G = p$ and $\reg G = r$. 
	By Lemma \ref{regz}, we only prove the case of $2 \le r$. 
	\begin{enumerate}
		\item The case $r \le p$. 
		
		We denote $r = 1 + r'$, $p = 1 + r' + a, d = 1 + r' + a + b$ where $a, b$ are some nonnegative integers. 
		
		The inequality $n - d - r \ge 1$ implies $n -2r' - 1 \ge 2 + a + b$.

		Therefore by Lemma \ref{regz}, there is a graph $G'$ with $n - 2r' - 1$ vertices such that 
		$\dim G' = 1 + a + b, \depth G' = 1 + a, \reg G' = 1$. 
		Then $G'$ has an independent set $S$ such that $\lvert S \rvert = a < \depth G'$. 
		
		We define $G$ as follows. 
		\begin{displaymath}
			\begin{aligned}
				V(G) = V(G') &\cup \{x_{1}, y_{1}, \cdots, x_{r'}, y_{r'}\} \cup \{v\}, \\
				E(G) = E(G') &\cup \{\{x_{1,}, y_{1}\}, \cdots, \{x_{r'}, y_{r'}\}\} \\ 
				             & \cup \{\{v, x_{1}\}, \cdots, \{v, x_{r'}\}\} \\ 
				             & \cup \{\{v, w\} : w \in V(G') \setminus S\}.  
			\end{aligned}
		\end{displaymath}
		
		By virtue of Lemma \ref{susp}, it can be ensured $\rvert V(G) \lvert = n, \dim G = 1 + r' + a + b = d$, 
		$\depth G = 1 + r' + a = p$ and $\reg G = 1 + r' = r$. 
		
		\item The case $r > p$. 
		
		Similarly denote $r = 1 + r'$, $p = 1 + a, d = 1 + r' + b$. 
		
		Repeat the above discussion. Since $n - 2r' - 1 \ge 2 + b$, there is a graph $G'$ with $n - 2r' - 1$ vertices such that 
		$\dim G' = 1 + b, \depth G' = 1, \reg G' = 1$ and $G'$ has an independent set $S$ with $\lvert S \rvert = a$. 
		
		We define $G$ as follows: 
		
		\begin{displaymath}
			\begin{aligned}
				V(G) = V(G') &\cup \{x_{1}, y_{1}, \cdots, x_{r'}, y_{r'}\} \cup \{v\}, \\
				E(G) = E(G') & \cup \{\{x_{1,}, y_{1}\}, \cdots, \{x_{r'}, y_{r'}\}\} \\ 
				             & \cup \{\{v, x_{1}\}, \cdots, \{v, x_{r'}\}\} \\ 
				             & \cup \{\{v, y_{1}\}, \cdots, \{v, y_{r' - a}\}\} \\ 
				             & \cup \{\{v, w\} : w \in V(G')\}. 
			\end{aligned}
		\end{displaymath}
		
		Then $\rvert V(G) \lvert = n, \dim G = 1 + r' + a = d$, 
		$\depth G = 1 + b = p$ and $\reg G = 1 + r' = r$. 
	\end{enumerate}
	
	Second we prove $\Graph_{\dim, \depth, \reg}(n) \subset C^{**}(n)$. 
	If $\reg(G) + \dim(G) = n$ then $\reg(G) = \m(G)$ because $n - \d(G) \ge \m(G) \ge \reg(G)$. 
	Theorem \ref{matching} says that if the above condition holds then $G$ is a pentagon or a star graph or a star triangle or a Cameron-Walker graph. Hence we obtain that if $\reg(G) + \dim(G) = n$ then $G$ is a star graph by virtue of Theorem \ref{cwdrdd} and easy calculation. Moreover if $G$ is a star graph with $n$ vertices, then 
	$\dim G = n - 1, \reg G = \depth G = 1$. 
	
	Thus we only prove that $\dim G$, $\depth G$, $\reg G$ and $n = \lvert V(G) \rvert$ satisfy the following inequality in $\dim(G) + \reg(G) \le n - 1$: 
	
	\[
	\dim G \le (n - \dim G - \reg G + 1)(\dim G - \depth G + 1) + (\reg G - 1). 
	\]
	
	By virtue of Lemma \ref{depmax}, it is sufficient to prove this inequality by induction on 
	$\lvert V(G) \rvert$. 
	
	We denote $\r(G)$ as $\reg G$. 
	
	Define $U = \{v \ | \ v \in V(G), \text{ there is a maximal independent set }S\text{ such that }\lvert S \rvert = \d(G), \text{ and }v \in S\}$ and $t = \min\{\lvert N(v) \rvert\ \ | \ v \in U\}$. 
	
	Suppose $v \in U$ and $\lvert N(v) \rvert = t$. If $v'$ is an isolated vertex of $G - v$, then $v'$ is only incident to $v$ and $\{v'\} \cup S \setminus v$ is also an independent set for any independent set $v \in S$. Then $\lvert N(v) \rvert = \lvert N(v') \rvert = 1$, thus $E(G) = \{v, v'\}$ , a contradiction to $n \geq 3$. Thus we may assume that there is no isolated vertex of $G - v$. 
	
	In addition if $v'$ is an isolated vertex of $G - N[v]$, $v'$ has $t$ incident vertices since $N(v') \subset N(v)$ and $\{v'\} \cup S \setminus v$ is also an independent set. Thus $N(v) = N(v')$. 
	
	Denote $m$ as a number of isolated vertices of $G - N[v]$, $G'$ as a graph $G - N[v]$ without isolated vertices and $a = \dim(G) - \dim(G - v) \in \{0, 1\}$. 

	\begin{enumerate}
		\item The case $\r(G) = \r(G - v) + 1$. 
		
		Assume $G - N[v]$ has an isolated vertex $v'$. Then $N(v) = N(v')$ by the above discussion, $\r(G) = \r(G - v)$ by virtue of Lemma \ref{twin}, a contradiction. Therefore $G - N[v]$ has no isolated vertex. 
		
		Then the following holds from $\dim(G) = \dim(G - N[v]) + 1 = \d(G') + 1$, $\r(G) = \r(G') + 1$ and 
		$\p(G) \le \p(G') + 1$ by Lemma \ref{depmax} and Proposition \ref{reg}.  
		\begin{displaymath}
			\begin{aligned}
				&\dim(G) = \d(G) = \d(G') + 1 \\
				&\le ((n - t - 1) - \d(G') - (\r(G') - 1))(\d(G') - \p(G') + 1) + \r(G') - 1 + 1 \\
				&\le (n - \d(G) - \r(G) + 1 + (1 - t))(\d(G) - \p(G) + 1) + \r(G) - 1 - 1 + 1 \\
				&\le (n - \d(G) - \r(G) + 1)(\d(G) - \p(G) + 1) + \r(G) - 1. 
			\end{aligned}
		\end{displaymath}
		
		\item The case $\r(G) = \r(G - v)$ and $\p(G) \le \p(G - v)$.  
		
		The followings holds. 
		\begin{displaymath}
			\begin{aligned}
				&\dim(G) = \d(G) = \d(G - v) + a \\ 
				&\le ((n - 1) - \d(G - v) - \r(G - v) + 1)(\d(G - v) - \p(G - v) + 1)+(\r(G - v) - 1) + a \\ 
				&= (n - \d(G) - \r(G) + 1 + (a - 1))(\d(G) - \p(G - v) + 1 - a) + (\r(G) - 1) + a \\
				&\le (n - \d(G) - \r(G) + 1)(\d(G) - \p(G) + 1) + (\r(G) - 1). 
			\end{aligned}
		\end{displaymath}
		
		\item The case  $\r(G) = \r(G - v)$ and $\p(G) = \p(G - v) + 1$. 
		
		Let $S \subset V(G)$ be an independent set of $G$ such that $\lvert S \rvert = \p(G)$ and 
		$v_1, \dots, v_m$ be isolated vertices of $G - N[v]$. 
		The condition $\p(G) = \p(G - v) + 1$ implies that $S \setminus \{v\}$ is a maximal independent set of $G - N[v]$ such that $\lvert S \setminus \{v\} \rvert = \p(G - v) = \p(G - N[v])$ and $v_i \in S$ for any $1 \le i \le m$ since $S$ is maximal. 
		Therefore $S \setminus \{v, v_1, \dots, v_m\} = S'$ is a maximal independent set of $G'$ and 
		$\lvert S' \rvert = \p(G') = \p(G) - m - 1$. 
		 
		By Lemma \ref{depmax}, $\p(G - \{v, v_1, \dots, v_m\}) \le \p(G') + t$ holds. Therefore 
		$\p(G) = \p(G - v) + 1 = \p(G') + m + 1 \le \p(G - \{v, v_1, \dots, v_m\}) \le \p(G') + t$. 
		
		Moreover by Proposition \ref{reg}, $\r(G - \{v, v_1, \dots, v_m\}) \le \r(G') + t$ and 
		$\r(G) = \r(G - \{v_1, \dots, v_m\}) \in \{\r(G') + 1, \r(G - \{v, v_1, \dots, v_m\})\}$. 
		Thus $\r(G) \le \r(G') + t$. 
		
		Then the following holds. 
		
		\begin{displaymath}
			\begin{aligned}
				&\dim(G) = \d(G) = \d(G') + m + 1 \\ 
				& \le ((n - t - 1 - m) - (\d(G')) - (\r(G') - 1))(\d(G') - \p(G') + 1) + (\r(G') - 1) + m + 1\\ 
				&= (n - \d(G) - (\r(G') + t) + 1)(\d(G) - \p(G) + 1) + (\r(G') + t) - 1 + (m + 1 - t)\\
			    &\le (n - \d(G) - \r(G) + 1)(\d(G) - \p(G) + 1) + \r(G) - 1 + (m + 1 - t) \\
			    &\le (n - \d(G) - \r(G) + 1)(\d(G) - \p(G) + 1) + \r(G) - 1. 
			\end{aligned}
		\end{displaymath}
	\end{enumerate}
\end{proof}


\begin{proof}[Proof of Corollary \ref{main}]
	
	Define $C^{**}(n, c) = \{(d, p) \in \NN^2 : 1 \le p \le d \le n - 1, 
	d \le (n - d - (c - 1))(d - p + 1) + (c - 1)\}$
	
	For any $n \geq 3$ and $ \lfloor \frac{n - 1}{2} \rfloor \ge .... \geq i \geq 2$, the inclusion $C^{**}(n,i) \subset C^{**}(n,i-1)$ is straightforward. Thus: 
	 
	\[
	C^{**}(n, \lfloor \frac{n - 1}{2} \rfloor) \subset C^{**}(n, \lfloor \frac{n - 1}{2} \rfloor - 1) \subset \cdots \subset C^{**}(n, 2) \subset C^{**}(n, 1) = C^{*}(n). 
	\]
	
	By definition of $\Graph_{\dim, \depth, \reg}(n)$, the following is obtained:
	\[
	\Graph_{\dim, \depth}(n) = \bigcup_{i = 1}^{\lfloor \frac{n - 1}{2} \rfloor} C^{**}(n, i) = C^{**}(n, 1) = C^*(n). 
	\]
	
\end{proof}

%
%
%
%
%


%

\bibliographystyle{plain}

\begin{thebibliography}{99}
	
	\bibitem{BH}
	W. Bruns, J. Herzog,
	{\it Cohen-Macaulay rings (Revised Edition)}. Cambridge University Press, 1998.
	
	\bibitem{CW}
	K.~Cameron and T.~Walker, 
	The graphs with maximum induced matching and maximum matching the same size, 
	{\em Discrete Math.} {\bf 299} (2005), 49--55. 
	
	\bibitem{DHS}
	H. Dao, C. Huneke, J. Schweig, 
	Bounds on the regularity and projective dimension of ideals associated to graphs, 
	{\em J. Algebraic Combin.}{\bf 38}(2013), 37--55. 
	
	\bibitem{DaoSchweig}
	H. Dao, J. Schweig,
	Projective dimension, graph domination parameters, and independence
	complex homology, {\em J. Combin.\  Theory, Ser.\  A} {\bf 120} (2013),
	453--469.
	
	
	\bibitem{CHA}
	C.A.Francisco, H.T. H\`{a}, A. Van Tuyl, 
	Splittings of monomial ideals, 
	{\em Proceedings of the American Mathematical Society} {\bf 137(10)} (2008), 3271-3282. 
	
	\bibitem{GS}
	D. Grayson, M. Stillman, 
	{\it Macaulay2, a software system for research in algebraic geometry.} 
	Available at {\tt http://www.math.uiuc.edu/Macaulay2/}
	
	
	\bibitem{HVT}
	H.T. H\`{a}, A. Van Tuyl, 
	Monomial ideals, edge ideals of hypergraphs, and their 
	graded Betti numbers, 
	{\em J. Algebraic Combin.} {\bf 27} (2008), 215--245.
	
	
	\bibitem{HHKO}
	T.~Hibi, A.~Higashitani, K.~Kimura and A.~B.~O'Keefe,
	Algebraic study on Cameron--Walker graphs, 
	{\em J. Algebra} {\bf 422} (2015), 257--269.
	
	\bibitem{HKM}
	T.~Hibi, H.~Kanno, K.~Matsuda, 
	Induced matching numbers of finite graphs and edge ideals, 
	{\em J. Algebra} {\bf 532} (2019), 311--322.  
	
	\bibitem{HKMT}
	T. Hibi, K. Kimura, K. Matsuda, A. Tsuchiya,
	Regularity and $a$-invariant of Cameron--Walker graphs. 
	J. Algebra 584, 215-242 (2021).
	
	\bibitem{HKMVT}
	T.~Hibi, K.~Kimura, K.~Matsuda, A.~Van Tuyl, 
	The regularity and $h$-polynomial of Cameron--Walker graphs.  
	Enumer. Comb. Appl. 2, No. 3, Article ID S2R17, 12 p. (2022).  
	
	
	
	
	
	
	\bibitem{HMVT} T. Hibi, K. Matsuda, A. Van Tuyl,
	Regularity and $h$-polynomials of edge ideals, 
	{\em Electron.\  J. Combin.} {\bf 26} (2019) \#P1.22 
	
	\bibitem{HKU} A. Higashitani, A. Kanno, R. Ueji, 
	Behaviors of pairs of dimensions and depths of edge ideals, 
	Comm. Algebra, 51 (8) (2023), 3574-3584.
	
	
	\bibitem{Katzman}
	M.~Katzman, 
	Characteristic-independence of Betti numbers of graph ideals, 
	{\em J. Combin.\  Theory, Ser.\  A} {\bf 113} (2006), 435--454.
	
	\bibitem{KumarKumarSarkar}
	A. Kumar, R. Kumar, R. Sarkar,
	Certain algebraic invariants of edge ideals of join of graphs. 
	J. Algebra Appl. 20, No. 6, (2021), 12 p. 
	
	
	\bibitem{Peeva}
	I. Peeva,
	{\it  Graded syzygies.}
	Algebra and Applications, 14. Springer-Verlag London, Ltd.,
	London, 2011.
	
	\bibitem{Rinaldo}
	G. Rinaldo, Some algebraic invariants of edge ideal of circulant graphs, 
	{\em Bull.\  Math.\  Soc.\  Sci.\  Math.\  Roumanie (N.S.)} {\bf 61 (109)}
	(2018), 95--105.
	
	\bibitem{SeyedFakhariYassemi}
	S. A. Seyed Fakhari, S. Yassemi,
	Improved bounds for the regularity of edge ideals of graphs, 
	{\em Collect.\  Math.} {\bf 69} (2018),
	249--262.
	
	\bibitem{TNT} 
	T.~N.~Trung, 
	Regularity, matchings and Cameron--Walker graphs, 
	{\em Collect.\  Math.} {\bf 71} (2020), 83--91. 
	
	\bibitem{V}
	W.~V.~Vasconcelos, 
	{\it Arithmetic of blowup algebras}, 
	London Mathematical Society Lecture Note Series, 
	{\bf 195}, 
	Cambridge University Press, Cambridge, 1994.   
	
	\bibitem{V2}
	W.~V.~Vasconcelos,  
	{\it Computational methods in commutative algebra and algebraic geometry}, 
	volume 2 of {\it Algorithms and Computation in Mathematics}. 
	Springer-Verlag, Berlin, 1998. 
	With chapters by David Eisenbud, Daniel R.~Grayson, J\"{u}rgen Herzog 
	and Michael Stillman. 
	
	
\end{thebibliography}

\begin{center}
	{\bf Acknowledgments} 
\end{center}

The author was supported by Grant-in-Aid for JSPS Research Fellow JP21J21603. 
	
\end{document}